\documentclass[12pt]{amsproc}
\usepackage{fullpage}
\usepackage{amsmath}
\usepackage{amsfonts}
\usepackage{amssymb}
\usepackage{amsthm}
\usepackage[utf8]{inputenc}
\usepackage{breqn}
\usepackage{afterpage}
\usepackage{longtable}
\usepackage{indentfirst}
\usepackage{caption}
\usepackage{subcaption}
\usepackage{graphicx}
\graphicspath{ {./images/} }
\newtheorem{theorem}{Theorem}[section]
\newtheorem{lemma}[theorem]{Lemma}
\newtheorem{corollary}[theorem]{Corollary}
\newtheorem{proposition}[theorem]{Proposition}

\theoremstyle{definition}
\newtheorem{definition}[theorem]{Definition}

\newtheorem{theorem-definition}[theorem]{Theorem-Definition}
\theoremstyle{remark}

\numberwithin{equation}{section}
\usepackage{color}
\usepackage{xcolor}

\begin{document}

\title{Capacity of the $\alpha$-Brjuno-R\"ussmann set}
\begin{author}[N.~Akramov]{Nurali Akramov}
    \address{National University of Uzbekistan,  Tashkent, Uzbekistan}
\email{nurali.akramov.1996$@$gmail.com}
\end{author}

\date{}

\begin{abstract} 
In this paper, we study a class of irrational numbers defined by certain conditions on the growth of the denominators in their continued fraction expansions. We prove that the complement of this class have zero $C_\sigma$-capacity with respect to a logarithmic kernel, where the value of $\sigma>0$ determined by the parameters of the condition. In particular, our approach includes classical Brjuno set and the $\alpha$-Brjuno-R\"ussmann set.
\end{abstract} 
\maketitle
\section{Introduction}

The $\alpha$-Bruno-Rüssmann condition was introduced by A. Bounemoura and J. F\'ejoz (see \cite{ABJF1}) and important in the study quasi-periodic Hamiltonian systems, particularly relevant when analyzing persistence of invariant tori under perturbation.

Consider a vector $\omega_0\in\mathbb{R}^n$ and define the function  $\Psi_{\omega_0}:[1,+\infty)\to[\Psi_{\omega_0}(1),+\infty]$ measures the size of the so-called small denominators 
\begin{align*}
 \Psi_{\omega_0}(Q)=\max{\{|k\omega_0|^{-1}:k\in\mathbb{Z}^n,0<|k|<Q\}}.
\end{align*}
where $k\omega_0=k^1\omega^1_0+k^2\omega^2_0+...+k^n\omega^n_0$.

 Now assume that $\omega_0$ is non-resonance that is $|k\omega_0|\not=0$ for any non zero $k\in\mathbb{Z}^n$. The function $\Psi_{\omega_0}$  is non-decreasing, piecewise constant, and has a countable number of discontinuities (see \cite{ABJF1}). Then, $\alpha$-Brjuno-R\"ussmann condition defined as follows:
\begin{definition}{(see \cite{ABJF1})}\label{ABR} 
A vector $\omega_0$ is said to satisfy $\alpha$-Brjuno-R\"ussmann condition if 
\begin{align}\label{aBR}
\int_1^{+\infty}\frac{\ln{\Psi_{\omega_0}(Q)}}{Q^{1+\frac{1}{\alpha}}}dQ<\infty, \text{ }\alpha\ge1.
\end{align}
\end{definition}
This condition prevents $\Psi_{\omega_0}$ from growing too fast at large $Q$. The set of such vectors is denote by $BR_\alpha$. This set becomes smaller as $\alpha$ increases and when $\alpha=1$ the condition reduces to the classic Brjuno-R\"ussmann condition i.e. $BR_1=BR\subset BR_{\alpha}$. 

In this area significant results obtained by H. R\"ussmann, C. Chavaudret, S. Marmi, S. Fischler, A. Bounemoura and J. F\'ejoz (see \cite{HRU}, \cite{CCSM}, \cite{ABJF2}, \cite{ABJF1}) and others. 

Now, let $\nu$ be the irrational number. This number can be expressed as a continues fraction:
\begin{align*}
\nu=v_0+\frac{1}{v_1+\cfrac{1}{v_2+...\cfrac{1}{v_n+...}}}=[v_0,v_1,v_2,...,v_n,...].
\end{align*}
The finite part of this continues fraction becomes a rational number
\begin{align}
\label{PnQn}[v_0,v_1,v_2,...,v_n]=\frac{P_n}{Q_n}.
\end{align} 

For the main result, we introduce the following condition on the sequence of denominators $Q_n$ of the continued fraction expansion of the irrational number $\nu$:

\begin{equation}\label{AlBr}
\sum_{n=1}^\infty\frac{\ln^\beta{Q_{n+1}}}{Q_n^\gamma}<+\infty 
\end{equation}
where $\beta,\gamma\in\mathbb{R}^+$.

We denote the set of such numbers by $\mathcal{A}(\beta,\gamma)$. This set is monotonic with respect to the parameters $\beta$ and $\gamma$. As $\beta$ increases or $\gamma$ decreases, the set  $\mathcal{A}(\beta,\gamma)$ contracts. Conversely, as $\beta$ decreases or $\gamma$ increases, the set  $\mathcal{A}(\beta,\gamma)$ expands.
When $\beta=1$ and $\gamma=1$, the condition reduces to the classic Brjuno condition in $\mathbb{C}$ (see \cite{ABR}). 

In earlier work, A. Sadullaev and K. Rakhimov (see \cite{AS1}) studied $C_\sigma$-capacity and $h$-Hausdorff measure of the complement of the Brjuno set in $\mathbb{C}$. Definition of these concepts are provided in Section 2. Our main results generalize their work by considering the condition \eqref{AlBr}. The relationship between the set $\mathcal{A}(\beta,\gamma)$  and the set $BR_\alpha$ established in Section 2.

\begin{theorem}\label{th1} 
Let $\beta>0$ and $\gamma>0$. If $\gamma\le2$, then the complement of the set $\mathcal{A}(\beta,\gamma)$ has zero $C_\sigma$-capacity with respect to the kernel  \begin{align*}
k_1(z,\xi) = |\ln|z - \xi||^{\frac{2\beta}{\gamma}}.
\end{align*}
In particular, it has zero $h$-Hausdorff measure with respect to the function $h(t)=\ln^{-\frac{2\beta}{\gamma}}\frac{1}{t}$.
 
 If $\gamma>2$, then the complement of the set $\mathcal{A}(\beta,\gamma)$ has zero $C_\sigma$-capacity with respect to the kernel  
\begin{align*}
k_2(z,\xi) = |\ln|z - \xi||^{\beta}.
\end{align*}
In particular, it has zero $h$-Hausdorff measure with respect to the function $h(t)=\ln^{-{\beta}}\frac{1}{t}$.
\end{theorem}

\section{Preliminary}
\subsection{Some properties of the continued fraction} Let $\nu$ be irrational number and the finite part of continued fraction of $\nu$ as a fraction $\frac{P_n}{Q_n}$ be as in \eqref{PnQn}. We will use the following properties of the sequence of fraction $\frac{P_n}{Q_n}$ for any $n\ge1$ (see \cite{SMA}).
\begin{enumerate}
\item  Denominator $Q_n$ satisfies 
\begin{align*}
Q_n>\frac{1}{2}\cdot \left(\frac{\sqrt{5}+1}{2}\right)^{n-1},
\end{align*}
so in particular sequence $Q_n$ grows exponentially.
\item We have classic two-sided inequality:
\begin{align*}
\frac{1}{2Q_nQ_{n+1}}<\left|\nu-\frac{P_n}{Q_n}\right|<\frac{1}{Q_nQ_{n+1}}.
\end{align*}
\end{enumerate}

\subsection{Connection between $BR_\alpha$ and $\mathcal{A}(\beta,\gamma)$}
We now prove the following proposition concerning the  complement of the $\alpha$-Brjuno R\"ussmann set and the complement of the $\mathcal{A}(\beta,\gamma)$ set.
\begin{proposition} If 
\begin{align*}
\sum_{n=1}^\infty\frac{\ln^\beta{Q_{n+1}}}{Q_n^\gamma}=+\infty,
\end{align*}
for $\beta=1$, $\gamma=1+\frac{1}{\alpha}$ and $\alpha\ge1$, then
\begin{align*}
\int_1^{+\infty}\frac{\ln{\Psi_{\omega_0}(Q)}}{Q^{1+\frac{1}{\alpha}}}dQ=+\infty.
\end{align*}
\end{proposition}

\begin{proof}
Let $\omega=(1,\nu)\in\mathbb{R}^2$ non-resonant and $\frac{P_n}{Q_n}$ is continued fraction of $\nu$. Equivalence of $\alpha$-Brjuno-R\"ussmann condition and 
\begin{align*}
\sum_{Q=1}^\infty\frac{\ln{\Psi_\omega(Q)}}{Q^{1+\frac{1}{\alpha}}}<+\infty
\end{align*}
is established in \cite{ABJF1}. 

Now assume that $\Psi(Q)$ defined as follows (see \cite{ABJF2})
\begin{align*}
    \Psi(Q)=|\nu Q_n-P_n|^{-1}, \text{ }Q_n\le Q\le Q_{n+1}.
\end{align*}

By Property 2 of the continued fractions, we have
\begin{align*}
    \sum_{Q=1}^\infty\frac{\ln{\Psi_\omega(Q)}}{Q^{1+\frac{1}{\alpha}}}=\sum_{Q=1}^\infty\frac{\ln{|\nu Q_n-P_n|^{-1}}}{Q^{1+\frac{1}{\alpha}}}>\sum_{\text{ }Q\not=Q_n}\frac{\ln{Q_{n+1}}}{Q^{1+\frac{1}{\alpha}}}+\sum_{n=1,\text{ }Q=Q_n}^\infty\frac{\ln{Q_{n+1}}}{Q_n^{1+\frac{1}{\alpha}}}>\sum_{n=1}^\infty\frac{\ln{Q_{n+1}}}{Q_n^{1+\frac{1}{\alpha}}}.
\end{align*}   
\end{proof}

\subsection{$C_\sigma$- capacity}
Capacity is one of the main concepts in potential theory. The capacity of the Brjuno set is the one of the connection points between potential theory and complex dynamical systems. 

Let  $K\Subset\{|z|<1\}$ be a compact set. Without loss of generality we may consider the kernel
{$k_\sigma(z, \xi) =|\ln|z-\xi||^{\sigma}$} with $\sigma>0$. Then, the potential $U^\mu(z)$ for the measure $\mu$ at a point $z\in\mathbb{C}$ is defined as
\begin{align*}
U^\mu(z)=\int_Kk_\sigma(z,\xi)d\mu(\xi),\text{ }\mu \in \mathring{M}_K^+
\end{align*}
where $\mathring{M}_K^+$ is a set of positive Borel probability measure supported on $K$ with $|\mu| = 1$.

Let 
\begin{align*}
I(\mu)=\int_KU^\mu(z)d\mu(z)=\iint_{K \times K} k_\sigma(z, \xi) \, d\mu(z) \, d\mu(\xi)
\end{align*}

and $W_\sigma(K)=\inf\{I(\mu):\mu\in\mathring{M}_K^+\}.$
Then, $C_\sigma$-capacity of the set $K$ is defined as

\begin{align*}
C_\sigma(K) =\frac{1}{W_\sigma(K)}=\left(\inf_{\mu \in \mathring{M}_K^+}\int_KU^\mu(z)d\mu(z)\right)^{-1}= \left( \inf_{\mu \in \mathring{M}_K^+} \iint_{K \times K} k_\sigma(z, \xi) \, d\mu(z) \, d\mu(\xi) \right)^{-1}.
\end{align*}

We note the following properties of the $C_\sigma$-capacity (see \cite{NSL2}, \cite{LCA})
\begin{enumerate}
    \item For any Borel set $E\subset\mathbb{C}^n$, the outer and inner $C_\sigma$-capacities coincides: \begin{align*}
    \overline{C}_\sigma(E)=\underline{C}_\sigma(E)=C_\sigma(E).
    \end{align*}
    \item $C_\sigma(E)=0$ if and only if there exists a finite Borel measure $\mu$ supported on $E$ such that its potential satisfies $U^\mu(z)\equiv+\infty$ for all $z\in{E}$.
    \item  If $C_\sigma(E)=0$, then the $h_{\delta}$-Hausdorff measure  of $E$ with  gauge function $h(t)=\log^{-\delta}{\frac{1}{t}}$ is zero for any $\delta>\sigma$ (see \cite{NSL1}).
    \item For any sequence of compact sets $K_j$,
\begin{align*}
C_\sigma\bigg(\bigcup_{j=1}^\infty K_j\bigg)\le\sum_{j=1}^{\infty}C_\sigma(K_j).
\end{align*}
\end{enumerate}
\subsection{$h$-Hausdorff measure}
Let  $h:[0,r_0]\to [0,+\infty)$ be a gauge function - continuous, strictly increasing function with $h(0)=0$ and $r_0>0$. For a bounded subset $E\subset\mathbb{R}^{n}$ and any  $0<\varepsilon<r_0$, consider a finite covering of $E$ by open balls $B_j(x_j, r_j)$ such that $r_j < \varepsilon$. Define 
\begin{align*}
H^{h}(E,\varepsilon)=\inf\left\{\sum_{j=1}^m h(r_j):\text{ } \bigcup_{j=1}^m{B_j}\supset E, \ \ r_j < \varepsilon\right\}
\end{align*}
for all $1 \le j \le m$, where $m$ depends on the chosen cover.
It is obvious that,  $H^{h}(E,\varepsilon)$ non-decreasing as $\varepsilon$ decreases. Hence, the limit
\begin{align*}
H^{h}(E)=\lim_{\varepsilon\to0+}H^{h}(E,\varepsilon)
\end{align*}
exists and defines the \textit{$h$-Hausdorff measure} of $E$.  When $h_{\delta}(t)=t^\delta$ for  $\delta>0$, it is called the classic $\delta$-dimensional Hausdorff measure of $E$.

\section{Capacity of the complement of the set $\mathcal{A}(\beta,\gamma)$}
Firstly, we need the following lemma.
\begin{lemma}\label{lemma} 
If
\begin{align*}
\sum_{n=1}^\infty\frac{\ln^\beta{Q_{n+1}}}{Q_n^\gamma}=+\infty
\end{align*}
where $\beta>0$ and $\gamma<2+\varepsilon$ for any $\varepsilon>0$. Then
\begin{align*}
\sum_{n=1}^\infty\frac{\ln^{\beta\frac{2+\varepsilon}{\gamma}}{Q_{n+1}}}{Q_n^{2+\frac{\varepsilon}{4}}}=+\infty.
\end{align*}
\end{lemma}
\begin{proof}
    
Let $\delta>0 $ be so small that $(1-\delta){(2+\varepsilon)}\ge2+\cfrac{\varepsilon}{4}$. Then, by using the H\"older inequality to the $\mathcal{A}(\beta, \gamma)$ condition, we have
\begin{equation}\label{Ho}
\sum_{n=1}^\infty\frac{\ln^\beta{Q_{n+1}}}{Q_n^\gamma}\le\left(\sum_{n=1}^\infty\frac{\ln^{{\beta\frac{2+\varepsilon}{\gamma}}}{Q_{n+1}}}{Q_n^{(1-\delta)(2+\varepsilon)}}\right)^{\frac{\gamma}{2+\varepsilon}}\left(\sum_{n=1}^\infty\frac{1}{Q_n^{\frac{\delta\gamma(2+\varepsilon)}{2+\varepsilon-\gamma}}}\right)^{\frac{2+\varepsilon-\gamma}{2+\varepsilon}}.
\end{equation}

By the property 1 of the continues fractions, it is obvious that the second series on the right-hand side converges. Since, $(1-\delta){(2+\varepsilon)}\ge2+\cfrac{\varepsilon}{4}$, we have

\begin{align*}
\sum_{n=1}^\infty\frac{\ln^{\beta\frac{2+\varepsilon}{\gamma}}{Q_{n+1}}}{Q_n^{(1-\delta){(2+\varepsilon)}}}<\sum_{n=1}^\infty\frac{\ln^{\beta\frac{2+\varepsilon}{\gamma}}{Q_{n+1}}}{Q_n^{{2+\frac{\varepsilon}{4}}}}.
\end{align*}
 
Therefore, by inequality \ref{Ho}, we have

\begin{align*}
\sum_{n=1}^\infty\frac{\ln^{\beta\frac{2+\varepsilon}{\gamma}}{Q_{n+1}}}{Q_n^{{2+\frac{\varepsilon}{4}}}}=+\infty.
\end{align*}

Hence, by the property 2 of the continued fractions

\begin{align*}
\sum_{n=1}^\infty\frac{\left|\ln^{\beta\frac{2+\varepsilon}{\gamma}}\left|\nu-\frac{P_n}{Q_n}\right|\right|}{Q_n^{2+\frac{\varepsilon}{4}}}\ge\sum_{n=1}^\infty\frac{\ln^{\beta\frac{2+\varepsilon}{\gamma}}{Q_nQ_{n+1}}}{Q_n^{2+\frac{\varepsilon}{4}}}>\sum_{n=1}^\infty\frac{\ln^{\beta\frac{2+\varepsilon}{\gamma}}{Q_{n+1}}}{Q_n^{2+\frac{\varepsilon}{4}}}=+\infty.
\end{align*}
\end{proof}
Now, we going to prove our main results.
\begin{proof}[Proof of Theorem \ref{th1}] 
To prove the main result by the second property of $C_\sigma$ capacity, it is enough to show that $U^\mu(\nu)=+\infty$ for any compact $ K\subset\mathbb{R}\setminus\mathcal{A}(\beta,\gamma)$ where $\beta>0$, $\gamma>0$. For this we use the following Borel measure
\begin{align}{\label{measure}}
\mu=\sum_{q=1}^\infty\sum_{p=1}^{q-1}\frac{\delta_{\frac{p}{q}}}{q^{2+\frac{\varepsilon}{4}}},\text{ }\varepsilon>0    
\end{align}
where $\delta_{\frac{p}{q}}$ is a Dirac measure supported at the point $\frac{p}{q}$. This measure is finite.

If $\gamma\le2$, the potential of $\mu$ with respect to the kernel $k_1(z,\xi)=|\ln|z-\xi||^{\frac{2\beta}{\gamma}}$ is

\begin{align*}
U^\mu_1(z)=\int{k_1(z,\xi)}d\mu(\xi)=\sum_{q=1}^\infty\sum_{p=1}^{q-1}\frac{1}{q^{2+\frac{\varepsilon}{4}}}\left|\ln{\left|z-\frac{p}{q}\right|}\right|^{\beta\frac{2+\varepsilon}{\gamma}}.
\end{align*}

We need to show that $U^\mu_1(z)=+\infty$ for any $z\in K$. Indeed, if $z\in K$ is rational number, it is clean. In the case, when $z$ is irrational number we will use the following
\begin{align*}
&U^\mu_1(\nu)=\sum_{q=1}^\infty\sum_{p=1}^{q-1}\frac{1}{q^{2+\frac{\varepsilon}{4}}}\left|ln\left|\nu-\frac{p}{q}\right|\right|^{\beta\frac{2+\varepsilon}{\gamma}}=\sum_{n=1}^\infty\frac{\left|\ln{\left|\nu-\frac{P_n}{Q_{n}}\right|}\right|^{\beta\frac{2+\varepsilon}{\gamma}}}{Q_n^{2+\frac{\varepsilon}{4}}}+\\
&\sum_{n=1}^\infty\sum_{p\not=P_n,\text{ }p=1}^{Q_n-1}\frac{1}{Q_n^{2+\frac{\varepsilon}{4}}}\left|ln\left|\nu-\frac{p}{Q_n}\right|\right|^{\beta\frac{2+\varepsilon}{\gamma}}+\sum_{q=1}^\infty\sum_{p\not=P_n,\text{ }p=1}^{q-1}\frac{1}{q^{2+\frac{\varepsilon}{4}}}\left|ln\left|\nu-\frac{p}{q}\right|\right|^{\beta\frac{2+\varepsilon}{\gamma}}\\
&\ge\sum_{n=1}^\infty\frac{\left|\ln{\left|\nu-\frac{P_n}{Q_{n}}\right|}\right|^{\beta\frac{2+\varepsilon}{\gamma}}}{Q_n^{2+\frac{\varepsilon}{4}}}.
\end{align*}

Therefore, by the lemma \ref{lemma} we can conclude that

\begin{align*}
U_1^\mu(\nu)=+\infty
\end{align*}
when $\beta>0$ and $\gamma<2+\varepsilon$ for any $\varepsilon>0$.

If $\gamma>2$, we will use the same measure \ref{measure} and the kernel 
\begin{align*}
k_2(z,\xi)=|\ln|z-\xi||^\beta.
\end{align*}
Then the potential is given by  
\begin{align*}
U^\mu_2(\nu)=\sum_{q=1}^\infty\sum_{p=1}^{q-1}\frac{1}{q^{2+\frac{\varepsilon}{4}}}\left|ln\left|\nu-\frac{p}{q}\right|\right|^{\beta}\ge\sum_{n=1}^\infty\frac{\left|\ln{\left|\nu-\frac{P_n}{Q_{n}}\right|}\right|^{\beta}}{Q_n^{2+\frac{\varepsilon}{4}}}.
\end{align*}

As in the previous proof, using property 2 of continued fraction, we obtain 
\begin{align*}
\sum_{n=1}^\infty\frac{\left|\ln{\left|\nu-\frac{P_n}{Q_{n}}\right|}\right|^{\beta}}{Q_n^{2+\frac{\varepsilon}{4}}}>\sum_{n=1}^\infty\frac{\ln^{\beta}{Q_{n+1}}}{Q_n^{2+\frac{\varepsilon}{4}}}>\sum_{n=1}^\infty\frac{\ln^{\beta}{Q_{n+1}}}{Q_n^\gamma}=+\infty.
\end{align*}
Therefore, $C_\sigma(K)=0$ for the set of real numbers that does not satisfy $\mathcal{A}(\beta, \gamma)$ condition for any $\beta>0$ and $\gamma>0$, so $C_\sigma(\mathbb{R} \setminus \mathcal{A}(\beta,\gamma)) = 0$.
The last assertion follows from property 4 of $C_\sigma$-capacity. 
\end{proof}
Special case of Theorem \ref{th1} when $\beta=1$ and $\gamma=1+\frac{1}{\alpha}$ for $\alpha\ge1$ is given by the following corollary.
\begin{corollary} The complement of the $\alpha$-Brjuno-R\"ussmann set has zero $C_\sigma$-capacity with respect to the kernel
$$ k_\sigma(z,\xi)=|\ln{|z-\xi|}|^{\frac{2\alpha}{\alpha+1}}$$
for any $\alpha\ge1$. In particular, it has zero $h$-Hausdorff measure with respect to the function $h(t)=\ln^{-\frac{2\alpha}{\alpha+1}}\frac{1}{t}$.
\end{corollary}
 
\end{document}